\documentclass{amsart}
\usepackage{amssymb}
\usepackage{amsthm}
\usepackage[curve,matrix,arrow]{xy}
\theoremstyle{plain}\newtheorem{Theorem}{Theorem}[section]
\theoremstyle{plain}
\theoremstyle{plain}
\theoremstyle{plain}\newtheorem{Lemma}[Theorem]{Lemma}
\theoremstyle{plain}\newtheorem{Proposition}[Theorem]{Proposition}
\theoremstyle{plain}
\theoremstyle{definition}
\theoremstyle{definition}\newtheorem{Example}[Theorem]{Example}
\theoremstyle{definition}
\theoremstyle{definition}
\theoremstyle{definition}\newtheorem{Remark}[Theorem]{Remark}
\theoremstyle{definition}

\def\CL{{\mathcal{L}}}

\def\Aut{\mathrm{Aut}}           \def\tenk{\otimes_k}

\def\chr{\mathrm{char}}
       
\def\Der{\mathrm{Der}}
\def\dim{\mathrm{dim}}           
\def\End{\mathrm{End}}           

\def\Ext{\mathrm{Ext}}

\def\HH{H\!H}
\def\Hom{\mathrm{Hom}}

\def\ker{\mathrm{ker}}           
             
\def\IDer{\mathrm{IDer}}

\def\ll{\ell\!\ell}

\def\op{\mathrm{op}}

\def\rad{\mathrm{rad}}

\def\sl{\mathfrak{sl}}
\def\soc{\mathrm{soc}}

\title{On the Lie algebra structure of $\HH^1(A)$ of a 
finite-dimensional algebra $A$} 
\author{Markus Linckelmann and Lleonard Rubio y Degrassi} 
 \date{}


\begin{document}

\begin{abstract}
Let $A$ be a split finite-dimensional associative unital algebra over 
a field. The first main result of this note shows that if the 
$\Ext$-quiver of $A$ is a  simple directed graph, then $\HH^1(A)$ is 
a solvable Lie algebra. The second main result shows that if the 
$\Ext$-quiver of $A$ has no loops and at most two parallel arrows in 
any direction, and if $\HH^1(A)$ is a simple Lie algebra, then 
$\chr(k)\neq$ $2$ and $\HH^1(A)\cong$ $\sl_2(k)$. The third result
investigates symmetric algebras with a quiver which has a vertex with
a single loop. 
\end{abstract}

\maketitle

\section{Introduction}

Let $k$ be a field. Our first result is a sufficient criterion for
$\HH^1(A)$ to be a solvable Lie algebra, where $A$ is a split
finite-dimensional $k$-algebra (where the term `algebra' without
any further specifications means an associative and
unital algebra).

\begin{Theorem} \label{Ext1-1dim-thm}
Let $A$ be a split finite-dimensional $k$-algebra. Suppose that the 
$\Ext$-quiver of $A$ is a simple directed graph. Then the derived Lie 
subalgebra of $\HH^1(A)$ is nilpotent;  in particular the Lie algebra 
$\HH^1(A)$ is solvable.
\end{Theorem}

The recent papers \cite{FR} and \cite{RSS} contain comprehensive
results regarding the solvability of $\HH^1(A)$ of tame algebras and 
blocks, and \cite{RSS} also contains a proof of Theorem 
\ref{Ext1-1dim-thm} with different methods. 
We will prove Theorem \ref{Ext1-1dim-thm} in Section 
\ref{proofs-Section} as part of the more precise Theorem 
\ref{Ext1-1dim-ll}, bounding the derived length of the Lie algebra 
$\HH^1(A)$ and the nilpotency class of the derived Lie subalgbra of 
$\HH^1(A)$ in terms of the Loewy length $\ll(A)$ of $A$. 
The hypothesis on the quiver of $A$ is equivalent to 
requiring that $\Ext_A^1(S,S)=$ $\{0\}$ for any simple $A$-module $S$ 
and $\dim_k(\Ext^1_A(S,T))\leq 1$ for any two simple $A$-modules $S$, 
$T$. If in addition $A$ is monomial, then Theorem \ref{Ext1-1dim-thm}
follows from work of Strametz \cite{Strametz}. 
The hypotheses on $A$ are not necessary for the derived Lie subalgebra 
of $\HH^1(A)$ to be nilpotent or for  $\HH^1(A)$ to be solvable; see 
\cite[Theorem 1.1]{BKL} or \cite{RSS} for examples. 

\medskip
The Lie algebra structure of $\HH^1(A)$ is invariant under derived 
equivalences, and for symmetric algebras, also invariant under stable 
equivalences of Morita type. 
Therefore, the conclusions of Theorem \ref{Ext1-1dim-thm} remain true
for any finite-dimensional $k$-algebra $B$ which is derived equivalent 
to an algebra $A$ satisfying the hypotheses of this theorem, or for a 
symmetric $k$-algebra $B$ which is stably equivalent of Morita type to 
a symmetric algebra $A$ satisfying the hypotheses of the theorem.

\medskip
If we allow up to two parallel arrows in the same direction in the 
quiver of $A$ but no loops, then it is possible for $\HH^1(A)$ to be 
simple as a Lie algebra. The only simple Lie algebra to arise in that 
case is $\sl_2(k)$, with $\chr(k)\neq$ $2$. 

\begin{Theorem} \label{Ext1-2dim-thm}
Let $A$ be a split finite-dimensional $k$-algebra. Suppose that
$\Ext_A^1(S,S)=$ $\{0\}$ for any simple $A$-module $S$ and that 
$\dim_k(\Ext^1_A(S,T))\leq 2$ for any two simple $A$-modules $S$, $T$.
If $\HH^1(A)$ is not solvable, then $\chr(k)\neq 2$ and 
$\HH^1(A)/\rad(\HH^1(A))$ is a direct product of finitely many copies
of $\sl_2(k)$. In particular, the following hold.

\begin{enumerate}
\item[\rm (i)]
If $\HH^1(A)$ is a simple Lie algebra, then $\chr(k)\neq 2$, and
$\HH^1(A)\cong$ $\sl_2(k)$. 

\item[\rm (ii)]
If $\chr(k)=2$, then $\HH^1(A)$ is a solvable Lie algebra.
\end{enumerate}
\end{Theorem}

This will be proved in Section \ref{proofs-Section}; for monomial
algebras this follows as before from Strametz \cite{Strametz}. 
An example of an algebra $A$ satisfying the hypotheses of this theorem 
is the Kronecker algebra, a $4$-dimensional $k$-algebra, with 
$\chr(k)\neq 2$, given by the 
directed quiver with two vertices $e_0$, $e_1$ and two parallel arrows 
$\alpha$, $\beta$ from $e_0$ to $e_1$. This example is a special case 
of more general results on monomial algebras; see in particular 
\cite[Corollary 4.17]{Strametz}.  As in the case of the previous
Theorem, the conclusions of Theorem \ref{Ext1-2dim-thm} remain true
for an algebra $B$ which is derived equivalent to an algebra $A$
satisfying the hypotheses of this theorem, or for a symmetric algebra
$B$ which is stably equivalent of Morita type to a symmetric algebra
$A$ satisfying the hypotheses of the theorem. 

We have the following partial result for symmetric algebras whose
quiver has a single loop at some vertex. 

\begin{Theorem} \label{loop}
Suppose that $k$ is algebraically closed. Let $A$ be a
finite-dimensional symmetric $k$-algebra, and let $S$ be a simple
$A$-module. Suppose that $\dim_k(\Ext_A^1(S,S))=$ $1$ and that for any
primitive idempotent $i$ in $A$ satisfying $iS\neq$ $\{0\}$ we have
$J(iAi)^2=$ $iJ(A)^2i$. If $\HH^1(A)$ is a simple Lie algebra, then
$\chr(k)=$ $p>2$ and $\HH^1(A)$ is isomorphic to either $\sl_2(k)$ or 
the Witt Lie algebra $W=$ $\Der(k[x]/(x^p))$.
\end{Theorem}

This will be proved in Section \ref{loopSection}, along with some 
general observations regarding the compatibility of Schur functors 
and the Lie algebra structure of $\HH^1(A)$. 
Section 5 contains some examples.

\section{On derivations and the radical}

We start with a brief review of some basic terminology. Let $k$ be a 
field. The {\it nilpotency class} of a nilpotent Lie algebra 
$\CL$ is the smallest positive integer $m$ such that $\CL^m=\{0\}$,
where $\CL^1=$ $\CL'$ and $\CL^{m+1}=$ $[\CL, \CL^m]$ for $m\geq 1$. 
In addition, the {\it derived length} of a solvable Lie algebra is 
the smallest positive integer $n$ such that $\CL^{(n)}=\{0\}$, where 
$\CL^{(1)}=$ $\CL'$ and $\CL^{(n+1)}=$ $[\CL^{(n)}, \CL^{(n)}]$ for 
$n\geq 1$.
A Lie algebra $\CL$ is called {\it strongly solvable} if its derived 
subalgebra is nilpotent. A Lie algebra $\CL$ of finite dimension $n$ 
is called {\it completely solvable} (also called {\it supersolvable}) 
if there exists a sequence of ideals 
$\CL_1 = \CL \supset \CL_2\supset \cdots \supset \CL_n  \supset \{0\}$
such that  $\dim_k(\CL_i)= n+ 1 -i$ for $1\leq i\leq n$.

\begin{Remark} \label{NS}
If $k$ is algebraically closed of characteristic zero, 
then the classes of strongly and completely solvable 
Lie algebras coincide with the class of  solvable Lie algebras as a 
consequence of Lie's theorem. Lie's theorem does not hold
in positive characteristic. If $k$ is algebraically closed of 
prime characteristic $p$, then by \cite[Theorem 3]{Dz}, a 
finite-dimensional Lie algebra $\CL$ over $k$ is strongly solvable 
if and only if $\CL$ is completely solvable.
\end{Remark}

Let $A$ be a finite-dimensional $k$-algebra. We denote by $\ell(A)$ 
the number of isomorphism classes of simple $A$-modules. The {\it 
Loewy length} $\ll(A)$ of $A$ is the smallest positive integer $m$ 
such that $J(A)^m=\{0\}$, where $J(A)$ denotes the Jacobson radical of 
$A$. We denote by $[A,A]$ the $k$-subspace of $A$ generated by the set 
of additive commutators $ab-ba$, where $a$, $b\in$ $A$. A {\it 
derivation on} $A$ is a $k$-linear map $f : A\to$ $A$ satisfying 
$f(ab)=$ $f(a)b+af(b)$ for all $a$, $b\in$ $A$. If $f$, $g$ are 
derivations on $A$, then so is $[f,g]=$ $f\circ g-g\circ f$, and the 
space $\Der(A)$ of derivations on $A$ becomes a Lie algebra in this way. 
If $c\in$ $A$, then the map $[c,-]$ defined by $[c,a]=$ $ca-ac$ is a 
derivation; any derivation of this form is called an {\it inner 
derivation}. 
The space $\IDer(A)$ of inner derivations is a Lie ideal in $\Der(A)$, 
and we have a canonical isomorphism $\HH^1(A)\cong$ $\Der(A)/\IDer(A)$;
see \cite[Lemma 9.2.1]{Weibel}. 
It is easy to see that any derivation on $A$ preserves the subspace 
$[A,A]$, and that any inner derivation of $A$ preserves any ideal in 
$A$. A finite-dimensional $k$-algebra $A$ is called {\it split} if
$\End_A(S)\cong$ $k$ for every simple $A$-module $S$.
If $A$ is split, then by the Wedderburn-Malcev Theorem, $A$ has a
separable subalgebra $E$ such that $A=$ $E\oplus J(A)$. Moreover, $E$
is unique up to conjugation by elements in the group $A^\times$ of
invertible elements in $A$. A primitive decomposition $I$ of $1$ in $E$
remains a primitive decomposition of $1$ in $A$.

For convenience, we mention the following well-known descriptions of
certain $\Ext^1$-spaces.

\begin{Lemma} \label{Ext1-Lemma1}
Let $A$ be a split finite-dimensional $k$-algebra,
let $i$ be a primitive idempotent in $A$. Set $S=$ $Ai/J(A)i$ and 
$S^\vee=$ $iA/iJ(A)$. We have $k$-linear isomorphisms 
$$\HH^1(A;S\tenk S^\vee) \cong \Ext^1_A(S,S)\cong
\Hom_A(J(A)i/J(A)^2i, S) \cong
\Hom_{A\tenk A^\op}(J(A)/J(A)^2, S\tenk S^\vee)\ .$$
\end{Lemma} 

\begin{Lemma} \label{Ext1-Lemma2}
Let $A$ be a split finite-dimensional $k$-algebra.
Let $i$ be a primitive idempotent in $A$, and set $S=$ $Ai/J(A)i$.
We have $\Ext^1_A(S,S)=$ $\{0\}$ if and only if $iJ(A)i\subseteq$ 
$J(A)^2$. 
\end{Lemma} 

\begin{proof}
By Lemma \ref{Ext1-Lemma1}, we have $\Ext^1_A(S,S)=$ $\{0\}$ if and 
only if $J(A)/J(A)^2$ has no simple bimodule summand isomorphic to
$S\tenk S^\vee$. This is equivalent to $i\cdot(J(A)/J(A)^2)\cdot i=$
$\{0\}$, hence to  $iJ(A)i\subseteq$ $J(A)^2$ as stated. 
\end{proof}

\begin{Lemma} \label{Ezero}
Let $A$ be a split finite-dimensional $k$-algebra, and let $E$ be
a separable subalgebra of $A$ such that $A=$ $E\oplus J(A)$. 
Every class in $\HH^1(A)$ has a representative $f\in$ $\Der(A)$
satisfying $E\subseteq$ $\ker(f)$.
\end{Lemma}

\begin{proof}
Let $f : A\to$ $A$ be a derivation. 
Since $E$ is separable, it follows that for any $E$-$E$-bimodule
$M$ we have $\HH^1(E; M)=$ $\{0\}$. In particular, the derivation 
$f|_E : E\to$ $A$ is inner; that is, there is an element $c\in$ $A$ 
such that $f(x)=$ $[c,x]$ for all $x\in$ $E$. Thus the derivation 
$f - [c,-]$ on $A$ vanishes on $E$ and represents the same class as 
$f$ in $HH^1(A)$. 
\end{proof}

\begin{Lemma} \label{iAjLemma}
Let $A$ be a split finite-dimensional $k$-algebra, and let $E$ be
a separable subalgebra of $A$ such that $A=$ $E\oplus J(A)$. 
Let $f : A\to$ $A$ be a derivation such that $E\subseteq$ $\ker(f)$.
For any two idempotents $i$, $j$ in $E$ we have $f(iAj)\subseteq$
$iAj$ and $f(AiAj)\subseteq$ $AiAj$.
\end{Lemma}

\begin{proof}
Let $i$, $j$ be idempotents in $E$, and let $a$, $b\in$ $A$. We have
$f(iaj)=$ $f(i^2aj)=$ $if(iaj)+$ $f(i)iaj=$ $if(iaj)$, since
$i\in$ $E\subseteq$ $\ker(f)$. Thus $f(iaj)\in$ $iA$. A similar
argument shows that $f(iaj)\in$ $Aj$, and hence $f(iaj)\in$ $iAj$. 
This shows the first statement. The second statement follows from this
and the equality $f(biaj)=$ $f(b)iaj+bf(iaj)$. 
\end{proof}

\begin{Lemma} \label{Ext1zero}
Let $A$ be a split finite-dimensional $k$-algebra such that
$\Ext^1_A(S,S)=\{0\}$ for all simple $A$-modules $S$. Then for any
derivation $f : A\to$ $A$ we have $f(J(A)) \subseteq$ $J(A)$. 
\end{Lemma}

\begin{proof}
Let $E$ be a separable subalgebra of $A$ such that $A=$ $E\oplus J(A)$.
Let $I$ be a primitive decomposition of $1$ in $E$ (hence also in $A$). 
Note that if $i$, $j\in$ $I$ are not conjugate in $A^\times$, then 
$iAj\subseteq$ $J(A)$. The hypotheses on $A$ imply that $J(A)i/J(A)^2i$ 
has no summand isomorphic to $Ai/J(A)i$, and hence that 
$iJ(A)i\subseteq$ $J(A)^2$ for any $i\in$ $I$. Then $iJ(A)j\subseteq$ 
$J(A)^2$ for any two $i$, $j\in$ $I$ which are conjugate in $A^\times$. 
Let now $f : A\to$ $A$ be a derivation. As noted above, any inner 
derivation preserves $J(A)$. Thus, by Lemma \ref{Ezero}, we may assume 
that $f|_E=0$. Since $J(A)=$ $\oplus_{i\in I}\ J(A)i$, it suffices to 
show that $f(J(A)i)\subseteq$ $J(A)i$, where $i\in$ $I$. If $j$ is 
conjugate to $i$, then $AjJ(A)i\subseteq$ $J(A)^2i$. Since $J(A)i=$ 
$\sum_{j\in I}\ AjJ(A)i$, it follows from Nakayama's Lemma that 
$J(A)i=$ $\sum_{j} AjAi$, where $j$ runs over the subset $I'$ of all 
$j$ in $I$ which are not conjugate to $i$. Now $f$ preserves 
the submodules $AjAi$ in this sum, thanks to Lemma \ref{iAjLemma}.
The result follows. 
\end{proof}

The following observations are variations of the statements in
\cite[Proposition 3.5]{LiRu1}.

\begin{Proposition} \label{rad-Prop}
Let $A$ be a split finite-dimensional $k$-algebra, and let $E$ be
a separable subalgebra of $A$ such that $A=$ $E\oplus J(A)$. 
For $m\geq 1$, denote by $D_m$ the subspace of $\Der(A)$ consisting
of all derivations $f : A\to$ $A$ such that $E\subseteq$ $\ker(f)$ and
such that $f(J(A))\subseteq$ $J(A)^m$. The following hold.

\begin{enumerate}
\item[\rm (i)]
For any positive integers $m$, $n$ we have $[D_m, D_n]\subseteq$
$D_{m+n-1}$. 

\item[\rm (ii)]
The space $D_1$ is a Lie subalgebra of $\Der(A)$, and for any positive 
integer $m$, the space $D_m$ is a Lie ideal in $D_1$. 

\item[\rm (iii)]
The space $D_2$ is a nilpotent ideal in $D_1$. More precisely, if
$\ll(A)\leq 2$, then $D_2=$ $\{0\}$, and if $\ll(A)>2$, then the
nilpotency class of $D_2$ is at most $\ll(A)-2$. 
\end{enumerate}
\end{Proposition}

\begin{proof}
The space of derivations on $A$ which vanish on $E$ is easily seen to
be closed under the Lie bracket on $\Der(A)$. Thus statement (i) 
follows from \cite[Lemma 3.4]{LiRu1}.  Statement (ii) is an immediate
consequence of (i). If $m\geq$ $\ll(A)$, then $J(A)^m=$ $\{0\}$, and
hence $D_m=$ $\{0\}$. Together with (i), this implies (iii). 
\end{proof}

\begin{Proposition} \label{AtoAJA2}
Let $A$ be a split finite-dimensional $k$-algebra, and let $E$ be a 
separable subalgebra of $A$ such that $A=$ $E\oplus J(A)$. For 
$m\geq 1$, denote by $D_m$ the subspace of $\Der(A)$ consisting of all 
derivations $f : A\to$ $A$ such that $E\subseteq$ $\ker(f)$ and such 
that $f(J(A))\subseteq$ $J(A)^m$. Suppose that every derivation $f$ on 
$A$ satisfies $f(J(A))\subseteq$ $J(A)$. Then the canonical algebra 
homomorphism $A\to$ $A/J(A)^2$ induces a Lie algebra homomorphism 
$\Phi : \HH^1(A)\to$ $\HH^1(A/J(A)^2)$, and the following hold. 

\begin{enumerate}
\item[\rm (i)] 
The canonical surjection $\Der(A)\to$ $\HH^1(A)$ maps $D_1$ onto 
$\HH^1(A)$. 

\item[\rm (ii)] 
The canonical surjection $\Der(A)\to$ $\HH^1(A)$ maps $D_2$ onto 
$\ker(\Phi)$; in particular, $\ker(\Phi)$ is a nilpotent ideal in 
the Lie algebra $\HH^1(A)$.

\item[\rm (iii)] 
The Lie algebra $\HH^1(A)$ is solvable if and only if
$\HH^1(A)/\ker(\Phi)$ is solvable.

\item[\rm (iv)] 
If the derived Lie algebra of $\HH^1(A)$ is contained in $\ker(\Phi)$,
then $\HH^1(A)$ is nilpotent. 

\item[\rm (v)]
If the Lie algebra $\HH^1(A)$ is simple, then $\Phi$ is injective.
\end{enumerate}
\end{Proposition}

\begin{proof}
The hypotheses on $\Der(A)$ together with Lemma \ref{Ezero} imply that 
$\HH^1(A)$ is equal to the image of the space $D_1$ in $\HH^1(A)$, 
whence (i). The canonical surjection $\Der(A)\to$ $\HH^1(A)$ clearly 
maps $D_2$ to $\ker(\Phi)$; we need to show the surjectivity of the 
induced map $D_2\to$ $\ker(\Phi)$. Note first that any inner derivation 
in $D_1$ is of the form $[c, -]$ for some $c$ which centralises $E$. 
Note further that the centraliser $C_A(E)$ of $E$ in $A$ is canonically 
isomorphic to $\Hom_{E\tenk E^\op}(E,A)$ (via the map sending an 
$E$-$E$-bimodule homomorphism $\alpha : E\to$ $A$ to $\alpha(1)$). 
Since $E$ is separable, hence projective as an $E$-$E$-bimodule, it 
follows that the functor $\Hom_{E\tenk E^\op}(E,-)$ is exact. In 
particular, the surjection $A\to$ $A/J(A)^2$ induces a surjection 
$C_A(E)\to$ $C_{A/J(A)^2}(E)$, where we identify $E$ with its image in 
$A/J(A)^2$. Let $f\in$ $D_1$ such that the class of $f$ is in 
$\ker(\Phi)$, or equivalently, such that the induced derivation, 
denoted $\bar f$, on $A/J(A)^2$ is inner. Then there is $c\in$ $A$ such 
that $\bar f=$ $[\bar c, -]$, where $\bar c=$ $c+J(A)^2$ centralises 
the image of $E$ in $A/J(A)^2$. By the above, we may choose $c$ such 
that $c$ centralises $E$ in $A$. Then the derivation $f-[c,-]$ 
represents the same class as $f$, still belongs to $D_1$, and induces 
the zero map on $A/J(A)^2$. Thus $f-[c,-]$ belongs in fact to $D_2$, 
proving (ii). The remaining statements are immediate consequences
of (ii).
\end{proof}

The next result includes the special case of Theorem \ref{Ext1-1dim-thm}
where $\ll(A)\leq 2$. 

\begin{Proposition} \label{Ext1-1dim}
Let $A$ be a split finite-dimensional $k$-algebra such that $J(A)^2=$
$\{0\}$. Suppose that for every simple $A$-module $S$ we have 
$\Ext^1_A(S,S)=$ $\{0\}$ and that for any two simple $A$-modules $S$, 
$T$ we have $\dim_k(\Ext_A^1(S,T))\leq 1$. Let $E$ be a separable
subalgebra of $A$ such that $A=$ $E\oplus J(A)$. The following hold

\begin{enumerate}
\item[{\rm (i)}]
If $A$ is basic and if $f$, $g$ are derivations on $A$ which vanish on 
$E$, then $[f,g]=$ $0$. 

\item[{\rm (ii)}]
The Lie algebra $\HH^1(A)$ is abelian.

\item[{\rm (iii)}]
Let $e(A)$ be the number of edges in the quiver of $A$. We have
$$\dim_k(\HH^1(A)) = e(A)-\ell(A) + 1\leq (\ell(A)-1)^2\ .$$ 

\end{enumerate}
\end{Proposition}

\begin{proof}
In order to prove (i), suppose that $A$ is basic. Let $I$ be a 
primitive decomposition of $1$ in $A$ such that $E=$ 
$\prod_{i\in I} ki$. Let $f$ and $g$ be derivations on $A$ which vanish 
on $E$. Then $f$, $g$ are determined by their restrictions to $J(A)$. 
By Lemma \ref{Ext1zero}, the derivations $f$, $g$ preserve $J(A)$. By 
the assumptions, each summand $iAj$ in  the vector space decomposition
$A=$ $\oplus_{i,j\in I}\ iAj$  has dimension at most one. By Lemma 
\ref{iAjLemma}, any derivation on $A$ which vanishes on $E$
preserves this decomposition. 
Therefore, if $X$ is a basis of $J(A)$ consisting of elements of
the subspaces $iAj$, $i$, $j\in$ $I$, which are nonzero, then 
$f|_{J(A)} : J(A)\to$ $J(A)$ is represented by a diagonal matrix. 
Similary for $g$. But then the restrictions of $f$ and $g$ to $J(A)$ 
commute. Since both $f$, $g$ vanish on $E$ , this implies that 
$[f,g]=0$, whence (i). If $A$ is basic, then clearly (i)  and
Lemma \ref{Ezero} together imply (ii). Since the hypotheses of the 
Lemma as well as the Lie algebra $\HH^1(A)$ are invariant under Morita 
equivalences, statement (ii) follows for general $A$.
In order to prove (iii), assume again that $A$ is basic. By the
assumptions, $e(A)=$ $\dim_k(J(A))=$ $|X|$. 
One verifies that the extension to $A$ by zero on $I$ of any linear 
map on $J(A)$ which preserves the summands $iAj$ (with $i\neq j$),
or equivalently, which preserves the one-dimensional spaces $kx$,
where $x\in$ $X$, is in fact a derivation. By Lemma \ref{Ezero}, any 
class in $\HH^1(A)$ is represented by such a derivation.
Thus the space of derivations on $A$ which vanish on $I$ is equal
to $\dim_k(J(A))=e(A)$. Each $i\in$ $I$ contributes an inner
derivation. The only $k$-linear combination of elements in $I$ 
which belongs to $Z(A)$ are the multiples of $1=$ $\sum_{i\in I} i$.
Thus the space of inner derivations which annihilate $I$ has
dimension $\ell(A)-1$, whence the first equality.
Since there are at most $\ell(A)-1$ arrows starting at any given
vertex, it follows that $e(A)\leq (\ell(A)-1)\ell(A)$, whence the
inequality as stated.
\end{proof}

The above Proposition can also be proved as a consequence of more 
general work of Strametz \cite{Strametz}, calculating the Lie algebra 
$\HH^1(A)$ for $A$ a split finite-dimensional monomial algebra. 

\section{Proofs of Theorems \ref{Ext1-1dim-thm} and \ref{Ext1-2dim-thm}}
\label{proofs-Section}

Theorem \ref{Ext1-1dim-thm} is a part of the following slightly more 
precise result. Let $k$ be a field.

\begin{Theorem} \label{Ext1-1dim-ll}
Let $A$ be a split finite-dimensional $k$-algebra. Suppose that for
every simple $A$-module $S$ we have $\Ext^1_A(S,S)=$ $\{0\}$ and that
for any two simple $A$-modules $S$, $T$ we have
$\dim_k(\Ext_A^1(S,T))\leq 1$. Set $\CL=$ $\HH^1(A)$, regarded as a Lie
algebra. 

\begin{enumerate} 
\item[\rm (i)]
If $\ll(A)\leq 2$ then $\CL$ is abelian.

\item[\rm (ii)]
If $\ll(A)>2$, then the derived Lie algebra $\CL'=$ $[\CL,\CL]$ is 
nilpotent of nilpotency class at most $\ll(A)-2$. The derived length
of $\mathcal{L}$ is at most $\log_2(\ll(A)-1)+1$.

\end{enumerate}

In particular, $\CL$ is solvable, and if $k$ is algebraically closed,
then $\CL$ is completely solvable. 
\end{Theorem}

\begin{proof}
If $\ll(A)\leq$ $2$, then $J(A)^2=\{0\}$, and hence (i) follows from
Proposition \ref{Ext1-1dim}. Suppose that $\ll(A)>2$. 
We may assume that $A$ is basic. Note that $A$ and $A/J(A)^2$ 
have the same $\Ext$-quiver, and hence we may apply Proposition 
\ref{Ext1-1dim} to the algebra $A/J(A)^2$; in particular, 
$\HH^1(A/J(A)^2)$ is abelian. Thus the kernel of the canonical Lie
algebra homomorphism $\CL=\HH^1(A)\to$ $\HH^1(A/J(A)^2)$ contains
$\CL'$. Proposition \ref{AtoAJA2} implies that $\CL'$ is contained in 
the image of $D_2$, hence nilpotent of nilpotency class at most 
$\ll(A)-2$ by Proposition  \ref{rad-Prop}. From the same proposition 
we have that if $f\in \mathcal{L}^{(n)}$, then 
$f(J(A))\subseteq J(A)^{2^{n-1}+1}$ for $n\geq 1$.  Therefore the 
derived length is at most $\log_2(\ll(A)-1)+1$. Since $\CL'$ is 
nilpotent, it follows that if $k$ is algebraically closed, then $\CL$ 
is completely solvable.
\end{proof}

\begin{proof}[{Proof of Theorem \ref{Ext1-2dim-thm}}]
By Lemma \ref{Ext1zero}, every
derivation $f : A\to$ $A$ preserves $J(A)$, and hence sends
$J(A)^2$ to $J(A)^2$. Thus the canonical map $A\to$ $A/J(A)^2$
induces a Lie algebra homomorphism $\varphi : \Der(A) \to$ 
$\Der(A/J(A)^2)$ which in turn induces a Lie algebra homomorphism
$\Phi : \HH^1(A)\to$ $\HH^1(A/J(A)^2)$. By Proposition 
\ref{AtoAJA2}, $\ker(\Phi)$ is a nilpotent ideal. 
If $\chr(k)=2$, then $\HH^1(A/J(A)^2)$ is solvable by 
\cite[Corollary 4.12]{Strametz}, and hence $\HH^1(A)$ is solvable.
Suppose now that $\HH^1(A)$ is not solvable. Then, by the above, we
have $\chr(k)\neq 2$. Then, by 
\cite[Corollary 4.11, Remark 4.16]{Strametz}, the Lie algebra 
$\HH^1(A/J(A)^2)$ is a finite direct product of copies of $\sl_2(k)$.
Thus $\HH^1(A)/\ker(\Phi)$ is a subalgebra of a finite direct product
of copies of $\sl_2(k)$, and hence $\HH^1(A)/\rad(\HH^1(A))$ is a
subquotient of a finite direct product of copies of $\sl_2(k)$.
Since any proper Lie subalgebra of $\sl_2(k)$ is solvable, it
follows easily that the semisimple Lie algebra 
$\HH^1(A)/\rad(\HH^1(A))$ is a finite direct product of copies of
$\sl_2(k)$.
\end{proof}

\section{Schur functors and proof of Theorem \ref{loop}}
\label{loopSection}

The hypothesis $J(iAi)^2=$ $iJ(A)^2i$ in the statement of Theorem
\ref{loop} means that for any primitive idempotent $j$ not conjugate 
to $i$ in $A$ we have $iAjAi\subseteq$ $J(iAi)^2$; that is, the image 
in $iAi$ of any path parallel to the loop at $i$ which is different 
from that loop is contained in $J(iAi)^2$. We start by collecting 
some elementary observations which will be used in the proof of 
Theorem \ref{loop}. Let $k$ be a field. 

\begin{Lemma} \label{AeAderivation}
Let $A$ be a $k$-algebra and $e$ an idempotent in $A$. Let 
$f : A\to$ $A$ be a derivation. The following hold.
\begin{enumerate}
\item[{\rm (i)}] We have $f(AeA)\subseteq$ $AeA$.
\item[{\rm (ii)}] We have $ef(e)e=$ $0$.
\item[{\rm (iii)}] We have $(1-e)f(e)(1-e)=$ $0$.
\item[{\rm (iv)}] We have $f(e)\in$ $eA(1-e)\oplus (1-e)Ae$.
\item[{\rm (v)}] We have $f(e) = [[f(e),e)], e]$; equivalently, the
derivation $f - [[f(e),e], -]$ vanishes at $e$.
\item[{\rm (vi)}] If $f(e)=0$, then for any $a\in$ $A$ we have
$f(eae)=$ $ef(a)e$; in particular, $f(eAe)\subseteq$ $eAe$ and
$f$ induces a derivation on $eAe$.
\item[{\rm (vii)}] If $f(e)=0$ and if $f$ is an inner derivation
on $A$, then $f$ restricts to an inner derivation on $eAe$.
\end{enumerate}
\end{Lemma}

\begin{proof}
Let $a$, $b\in$ $A$. Then $aeb=$ $aeeb$, hence $f(aeb)=$
$aef(eb)+f(ae)eb\in$ $AeA$, implying the first statement.
We have $f(e)=$ $f(e^2)=$ $f(e)e+ef(e)$. Right multiplication of this
equation by $e$ yields $f(e)e=$ $f(e)e+ef(e)e$, whence the second
statement. Right and left multiplication of the same equation by $1-e$
yields the third statement. Statement (iv) follows from combining
the statements (ii) and (iii). We have
$[[f(e),e], e]=$ $[f(e)e-ef(e), e]$. Using that $ef(e)e=0$ this
is equal to $f(e)e+ef(e)=f(e)$, since $f$ is a derivation. This
shows (v). Suppose that $f(e)=0$. Let $a\in$ $A$. Then $f(eae)=$
$f(e)ae+ef(a)e+eaf(e)=$ $ef(a)e$, whence (vi). If in addition
$f=$ $[c, -]$ for some $c\in$ $A$, then the hypothesis $f(e)=0$
implies that $ec=ce$, and hence (vi) implies that the restriction
of $f$ to $eAe$ is equal to the inner derivation $[ce, -]$. This
completes the proof of the Lemma.
\end{proof}

\begin{Proposition} \label{der-Schur}
Let $A$ be a $k$-algebra, and let $e$ be an
idempotent in $A$. For any derivation $f$ on $A$ satisfying $f(e)=0$
denote by $\varphi(f)$ the derivation on $eAe$ sending $eae$ to
$ef(a)e$, for all $a\in$ $A$. The correspondence $f\mapsto \varphi(f)$
induces a Lie algebra homormophism $HH^1(A)\to$ $\HH^1(eAe)$.
If $A$ is an algebra over a field of prime characteristic $p$, then
this map is a homomorphism of $p$-restricted Lie algebras.
\end{Proposition}

\begin{proof}
Let $f$ be an arbitrary derivation on $A$. By Lemma
\ref{AeAderivation} (v), the derivation $f-[[f(e), e], - ]$ vanishes
at $e$. Thus every class in $\HH^1(A)$ has a representative
in $\Der(A)$ which vanishes at $e$. By Lemma \ref{AeAderivation} (vi),
any derivation on $A$ which vanishes at $e$ restricts to a derivation
on $eAe$, and by Lemma \ref{AeAderivation} (vii), this restriction
sends inner derivations on $A$ to inner derivations on $eAe$, hence
induces a map $\HH^1(A)\to$ $\HH^1(eAe)$. A trivial verification
shows that if $f$, $g$ are two derivations on $A$ which vanish at $e$,
the so does $[f,g]$, and an easy calculation shows that therefore
the above map $\HH^1(A)\to$ $\HH^1(eAe)$ is a Lie algebra
homomorphism. If $A$ is an algebra over a field of characteristic
$p>0$, and if $f$ is a derivation on $A$ which vanishes at $e$, then
the derivation $f^p$ vanishes on $e$ and the restriction to $eAe$
commutes with taking $p$-th powers by Lemma \ref{AeAderivation} (vi).
This shows the last statement.
\end{proof}

We call the Lie algebra homomorphism $\HH^1(A)\to$ $\HH^1(eAe)$ in
Proposition \ref{der-Schur} the {\it canonical Lie algebra
homomorphism} induced by the Schur functor given by multiplication
with the idempotent $e$.

For $A$ a finite-dimensional $k$-algebra and $m$ a positive integer, 
denote by $\HH^1_{(m)}(A)$ the subspace of $\HH^1(A)$ of classes which
have a a representative $f\in$ $\Der(A)$ satisfying 
$f(J(A))\subseteq$ $J(A)^m$.

\begin{Proposition} \label{der1-Schur}
Let $A$ be a split finite-dimensional $k$-algebra. Let 
$i$ be a primitive idempotent in $A$. Set $S=$ $Ai/J(A)i$. Suppose 
that $\Ext^1_A(S,S) =$ $\{0\}$. Then image of the canonical map
$\HH^1(A)\to$ $\HH^1(iAi)$ is contained in $\HH^1_{(1)}(iAi)$.
\end{Proposition}

\begin{proof}
By Lemma \ref{Ext1-Lemma2} we have $iJ(A)i=$ $iJ(A)^2i$.
By Lemma \ref{AeAderivation} (v), any class
in $\HH^1(A)$ is represented by a derivation $f$ satisfying $f(i)=0$.
Thus if $a\in$ $J(A)$, then $iai=$ $ibci$ for some $b$, 
$c\in$ $J(A)$, and hence $f(iai)=$ $if(b)ci+ibf(c)i\in$ $iJ(A)i$. 
\end{proof}

\begin{Proposition}  \label{der-Schur-nonzero}
Let $A$ be a split symmetric $k$-algebra. 
Let $i$ be a primitive idempotent in $A$. Set $S=$ $Ai/J(A)i$.
Suppose that $\Ext^1_A(S,S)\neq$ $\{0\}$. Then the canonical Lie
algebra homomorphism $\HH^1(A)\to$ $\HH^1(iAi)$ is nonzero.
\end{Proposition}

\begin{proof}
Set $S^\vee=$ $iA/iJ(A)$. Choose a maximal semisimple subalgebra
$E$ of $A$. Since $\Ext^1_A(S,S)$ is nonzero, it follows from Lemma 
\ref{Ext1-Lemma1} that $J(A)/J(A)^2$ has a direct summand isomorphic to 
$S\tenk S^\vee$ as an $A$-$A$-bimodule. Since $A$ is symmetric, we have
$\soc(A)\cong$ $A/J(A)$, and hence $\soc(A)$ has a bimodule summand 
isomorphic to $S\tenk S^\vee$. Thus there is a bimodule homomorphism 
$J(A)/J(A)^2\to$ $\soc(A)$ with image isomorphic to $S\tenk S^\vee$. 
Composing with the canonical map $J(A)\to$ $J(A)/J(A)^2$ yields a 
bimodule homomorphism $f : J(A)\to$ $\soc(A)$ with kernel containing 
$J(A)^2$ and with image isomorphic to $S\tenk S^\vee$. 
Extending $f$ by zero on $E$ yields a derivation $\hat f$ on $A$, by 
Lemma \ref{Ezero}. Restricting $\hat f$ to $iJ(A)i$ sends $iJ(A)i$ 
to a nonzero subspace of $\soc(A)$ isomorphic to $iS\tenk S^\vee i$, 
hence onto $\soc(iAi)$. Thus the image of $\hat f$ under the canonical 
map $\Der(A)\to$ $\Der(iAi)$ from Proposition \ref{der-Schur} is a 
nonzero derivation with kernel containing $ki+J(iAi)^2$ and image in 
$\soc(iAi)$. By \cite[Corollary 3.2]{BKL}, the class in $\HH^1(iAi)$ 
of this derivation is nonzero, whence the result.
\end{proof}

\begin{Proposition} \label{Wittsubalgebras}
Let $p$ be an odd prime and suppose that $k$ is algebraically closed
of characteristic $p$. Set $W=$ $\Der(k[x]/(x^p))$. For 
$-1\leq i\leq p-2$ let $f_i$ be the derivation of $k[x]/(x^p)$ sending 
$x$ to $x^{i+1}$, where we identify $x$ with its image in $k[x]/(x^p)$. 
Let $L$ be a simple Lie subalgebra of $W$. Then either $L=W$, or 
$L\cong$ $\sl_2(k)$. 
\end{Proposition}

\begin{proof}
Note that the subalgebra $S$ of $W$ spanned by the $f_i$ with
$0\leq i\leq p-2$ is solvable. Thus $L$ is not contained in $S$.
Note further that $\dim_k(L)\geq 3$. Therefore there exist derivations
$$f = \sum_{i=-1}^{p-1} \lambda_i f_i$$
$$g= \sum_{i=t}^{p-2} \mu_i f_i$$
belonging to $L$ with $\lambda_{-1}=1$, and $\mu_t=1$, where $t$ is an 
integer such that $0\leq t\leq p-2$. Choose $g$ such that $t$ is 
minimal with this property. But then $[f,g]$ belongs to $L$. Since 
$[f_{-1}, f_t]=$ $(t+1)f_{t-1}$, the minimality of $t\geq 0$ forces 
$t=0$; that is we have
$$g= \sum_{i=0}^{p-2} \mu_i f_i$$
and $\mu_0=1$. Since $\dim_k(L)\geq 3$, it follows that there is
a third element $h$ in $L$ not in the span of $f$, $g$, and hence,
after modifying $h$ by a linear combination of $f$ and $g$, we can 
choose $h$ such that 
$$h = \sum_{i=s}^{p-2} \nu_i f_i$$
for some $s$ such that $1\leq s\leq p-2$ and $\nu_s=1$. Choose
$h$ such that $s$ is minimal with this property. Again by
considering $[f,h]$, one sees that the minimality of $s$ forces
$s=1$. If $L$ is $3$-dimensional, then $L\cong$ $\sl_2(k)$, where
we use that $k$ is algebraically closed.
If $\dim_k(L)\geq 4$, then $L$ contains an element of the form
$$u = \sum_{i=r}^{p-2} \tau_i f_i$$
with $2\leq r\leq p-2$ and $\tau_r=1$. But then applying
$[f,-]$ and $[h,-]$ repeatedly to $u$ shows that $L$ contains
a basis of $W$, hence $L=W$.
\end{proof}

\begin{Remark}
Note that if $\chr(k)=p>2$, then the Witt Lie algebra $W$ contains
indeed a subalgebra isomorphic to $\sl_2(k)$. Let $\mathfrak{f},
\mathfrak{e}, \mathfrak{h}$ be elements of the basis of $\sl_2(k)$ 
such that $[\mathfrak{e},\mathfrak{f}]=\mathfrak{h}$, 
$[\mathfrak{h},\mathfrak{f}]=-2\mathfrak{f}$, and 
$[\mathfrak{h},\mathfrak{e}]=2\mathfrak{e}$.
Then we have a Lie algebra  isomorphism $\sl_2(k)\cong$ 
$\langle f_{-1}, f_0, f_1\rangle$ sending  
$\mathfrak{f}$ to $f_{-1}$, $\mathfrak{h}$ to $2f_{0}$,   
and $\mathfrak{e}$ to $-f_1$.
\end{Remark}

\begin{proof}[Proof of Theorem \ref{loop}]
We use the notation and hypotheses of the notation in Theorem
\ref{loop}, and we assume that the Lie algebra $\HH^1(A)$ is
simple. We show that this forces $\HH^1(A)$ to be a Lie
subalgebra of the Witt Lie algebra $W$ with $\chr(k)=p>2$, and
then the result follows from Proposition \ref{Wittsubalgebras}.

Since $\HH^1(A)$ is simple and since $\Ext_A^1(S,S)$ is nonzero, it 
follows from  Proposition \ref{der-Schur-nonzero} that the canonical
Lie algebra homomorphism  $\Phi: \HH^1(A) \to \HH^1(iAi)$ from
Proposition \ref{der-Schur} is injective. 
By the assumptions, $iAi$ is a local algebra whose quiver has
only one loop. Therefore $A\cong k[x]/(v)$ for some polynomial 
$v \in k[x]$ of degree at least $1$. Since $k$ is algebraically 
closed, $v$ is a product of powers of linear polynomials, say 
$\prod_i (x-\beta_i)^{n_i}$, with pairwise distinct $\beta_i$ and
positive integers $n_i$. Therefore $\HH^1(iAi)$ is a  direct product 
of the Lie algebras corresponding to these factors. It follows that 
$\HH^1(A)$ is isomorphic to a Lie subalgebra of 
$\HH^1(k[x]/((x-\beta)^{n}))$ for some positive integer $n$. 
After applying the automorphism $x\mapsto x+\beta$ of $k[x]$ 
we have that $\HH^1(A)$ is isomorphic to a Lie subalgebra of 
$\HH^1(k[x]/(x^n))$ for some positive integer $n$. 
If $\chr(k)=p$ does not divide $n$ or if $\chr(k)=0$, then the linear 
map sending $x$ to $1$ is not a derivation on $k[x]/(x^n)$, and 
therefore $\HH^1(k[x]/(x^n))$ is solvable in that case. 
Since Lie subalgebras of solvable Lie algebras are solvable, this
contradicts the fact that $\HH^1(A)$ is simple. 
Thus we have $\chr(k)=p>0$ and $n=$ $pm$ for some positive integer $m$.
Since $\chr(k)=p$, it follows that the canonical surjection 
$k[x]/(x^n)\to$ $k[x]/(x^p)$ induces a Lie algebra homomorphism 
$\HH^1(k[x]/(x^n))\to$ $W=$ $\HH^1(k[x]/(x^p))$ with a nilpotent
kernel. Thus $\HH^1(A)$ is not containd in that kernel, and hence
$\HH^1(A)$ is isomorphic to a Lie subalgebra of $W$. 
The result follows.
\end{proof}

To conclude this section we note that although it is not clear which 
simple Lie algebras might occur as $\HH^1(A)$ when $\Ext^1_{A}(S,S)=$
$\{0\}$ for all simple $A$-modules $S$, it easy to show that  
$\HH^{\ast}(A)$ is not a simple graded Lie algebra (with respect to
the Gerstenhaber bracket).

\begin{Proposition}
Let $A$ be a finite dimensional $k$-algebra, and assume that for 
every simple $A$-module $S$ we have $\Ext^1_A(S,S)=\{0\}$. Then
$\HH^{\ast}(A)$ is not a perfect graded Lie algebra. In particular, 
$\HH^{\ast}$ is not simple. 
\end{Proposition}

\begin{proof}
If $f\in C^1(A,A):=\Hom_k(A,A)$ and  if $g \in C^{0}(A,A):=\Hom_k(k,A)$,
then the Gerstenhaber bracket is given by $[f,g]=f(g)$, i.e. simply 
evaluating $f$ in $g$. Note that $1\in Z(A)=HH^0(A)$.  By Lemma 
\ref{Ezero} and Lemma \ref{Ext1zero}, $f$ preserves $J(A)$ and 
we may assume $E\subseteq \ker (f)$. Therefore the derived Lie 
subalgebra of $\HH^{\ast}(A)$ does not contain $1_A$. 
\end{proof}

\begin{Remark} 
Lemma \ref{AeAderivation} and Proposition \ref{der-Schur} hold
for algebras over an arbitrary commutative ring instead of $k$.
\end{Remark}

\section{Examples} 

Theorem \ref{Ext1-1dim-thm} applies to certain blocks of symmetric
groups. 

\begin{Proposition} \label{symmgroup}
Let $k$ be a field of prime characteristic $p$. Let $A$ be a defect 
$2$ block of a symmetric group algebra $kS_n$ or the principal block 
of $kS_{3p}$. Then $\HH^1(A)$ is a solvable  Lie algebra. 
\end{Proposition}

\begin{proof}
From \cite[Theorem 1]{Scopes} and from \cite[Theorem 5.1]{Martin} we 
have that the simple modules do not self-extend and the $\Ext^1$-space 
between two simple modules is at most one-dimensional. The statement 
follows from Theorem  \ref{Ext1-1dim-thm}.
\end{proof}

\begin{Remark}
A conjecture by Kleshchev and Martin predicts that simple 
$kS_n$-modules in odd characteristic do not admit self-extensions. 
\end{Remark}

\begin{Proposition} \label{tamesymm}
Let $A$ be a tame symmetric algebra over a field $k$ with $3$ 
isomorphism classes of simple modules of type $3\mathcal{A}$ or 
$3\mathcal{K}$. Then $\mathrm{HH}^1(A)$ is a solvable Lie algebra.
\end{Proposition}

\begin{proof}
From the list at the end of Erdmann's book \cite{Erd} we have that the 
simple modules in these cases do not self-extend and that the 
$\Ext^1$-space between two simple modules is at most one-dimensional. 
The statement follows from Theorem \ref{Ext1-1dim-thm}.
\end{proof}

As mentioned in the introduction, the above Proposition is part of
more general results on tame algebras in \cite{FR} and \cite{RSS}. 
We note some other examples of algebras whose simple modules do not
have nontrivial self-extensions.

\begin{Theorem} [{\cite[Theorem 3.4]{BNP}}]
Let $G$ be a connected semisimple algebraic group defined and split 
over the field $\mathbb{F}_p$ with $p$ elements, and $k$ be an 
algebraic closure of $\mathbb{F}_p$. Assume $G$ is almost simple and 
simply connected and let $G(\mathbb{F}_q)$ be the finite Chevalley 
group consisting of $\mathbb{F}_q$-rational points of G where 
$q= p^r$ for a non-negative integer $r$. Let $h$ be the Coxeter 
number of $G$. For  $r\geq 2$ and $p\geq 3(h-1)$, we have
$\Ext^1_{kG(\mathbb{F}_q)}(S,S)=\{0\}$ for every simple 
$kG(\mathbb{F}_q)$-module $S$.
\end{Theorem}
  
\begin{Remark}
Let $G$ be a simple algebraic group over a field of characteristic  
$p>3$, not of  type $A_1, G_2$ and $F_4$. Proposition 1.4 in \cite{TZ} 
implies that not having self-extensions does not allow to lift to 
characteristic zero certain simple modular representations.  Therefore, 
for these cases the Lie structure of $\HH^1$ plays a central role.  
\end{Remark}  

In the context of blocks with abelian defect groups one expects
(by Brou\'e's abelian defect conjecture) every block of a finite
group algebra with an abelian defect group $P$ to be derived
equivalent to a twisted group algebra of the form 
$k_\alpha(P\rtimes E)$, where $E$ is the inertial quotient of the
block and where $\alpha$ is a class in $H^2(E; k^\times)$, inflated to
$P\rtimes E$ via the canonical surjection $P\rtimes E\to$ $E$. 
Thus the following observation is relevant in cases where 
Brou\'e's abelian defect conjecture is known to hold (this includes
blocks with cyclic and Klein four defect).

\begin{Proposition} \label{kPEnoloops}
Let $k$ be a field of prime characteristic $p$. Let $P$ be a 
finite $p$-group and $E$ an abelian $p'$-subgroup of 
$\Aut(P)$ such that $[P,E]=P$. Set $A=$ $k(P\rtimes E)$. Suppose that 
$k$ is large enough for $E$, or equivalently, that $A$ is split. For 
any simple $A$-module $S$ we have $\Ext^1_A(S,S)=$ $\{0\}$. 
\end{Proposition}

\begin{proof}
Since $E$ is abelian, it follows that $\dim_k(S)=1$, and hence
that $S\tenk-$ is a Morita equivalence. This Morita equivalence
sends the trivial $A$-module $k$ to $S$, hence induces an isomorphism
$\Ext_A^1(k,k)\cong$ $\Ext_A^1(S,S)$. It suffices therefore
to show the statement for $k$ instead of $S$. That is, we need
to show that $H^1(P\rtimes E; k)=$ $\{0\}$, or equivalently, that
there is no nonzero group homomorphism from $P\rtimes E$ to the
additive group $k$. Since $[P,E]=P$, it follows that every abelian 
quotient of $P\rtimes E$ is isomorphic to a quotient of $E$, hence 
has order prime to $p$. The result follows.
\end{proof}

\begin{Example}
If $B$ is a block of a finite group algebra over an algebraically 
closed field $k$ of characteristic $p>0$ with a nontrivial cyclic 
defect group $P$ and nontrivial inertial quotient $E$, then 
$\HH^1(B)$ is a solvable Lie algebra, isomorphic to $\HH^1(kP)^E$, 
where $E$ acts on $\HH^1(kP)$ via the group action of $E$ on $P$.
Indeed, $B$ is derived equivalent to the Nakayama algebra
$k(P\rtimes E)$, which satisfies the hypotheses of Theorem 
\ref{Ext1-1dim-thm} (thanks to the assumption $E\neq$ $1$, which
implies $[P,E]=P$). 
Note that $kP$ is isomorphic to the truncated  polynomial algebra 
$k[x]/(x^{p^d})$, where $p^d=$ $|P|$. 
\end{Example}

\bigskip\bigskip

\noindent\smallskip
{\it Acknowledgements.} 
The second author has been  supported by the projects,  
``Oberwolfach Leibniz Fellows'', DAAD Short-Term Grant
(57378443) and 19880/GERM/15. He also thanks the 
Department of Mathematics of the University of Leicester.


\end{document}